\documentclass[11pt]{article}
\usepackage{amscd}
\usepackage{amsfonts}
\usepackage{amsmath}
\usepackage{amssymb}
\usepackage{amsthm}
\usepackage{bbm}
\usepackage{CJK}
\usepackage{fancyhdr}
\usepackage{graphicx}
\usepackage{hyperref}
\usepackage{indentfirst}
\usepackage{latexsym}
\usepackage{mathrsfs}
\usepackage{xypic}

\newtheorem{theorem}{Theorem}[section]
\newtheorem{lemma}[theorem]{Lemma}
\newtheorem{definition}[theorem]{Definition}
\newtheorem{proposition}[theorem]{Proposition}
\newtheorem{example}[theorem]{Example}

\newtheorem{remark}[theorem]{Remark}

\usepackage[top=1in,bottom=1in,left=1.25in,right=1.25in]{geometry}
\def\<{\langle}
\def\>{\rangle}

\def\c{\cdot}

\date{}
\begin{document}
\renewcommand{\baselinestretch}{1.2}
\renewcommand{\arraystretch}{1.0}
\title{\bf On equivariant Lie-Yamaguti algebras and related structures}
\author{{\bf Shuangjian Guo$^{1}$, Bibhash Mondal$^{2}$,     Ripan Saha$^{3}$\footnote
        { Corresponding author (Ripan Saha),  Email: ripanjumaths@gmail.com}}\\
{\small 1. School of Mathematics and Statistics, Guizhou University of Finance and Economics} \\
{\small  Guiyang  550025, P. R. of China}\\
  {\small 2. Department of Mathematics, Behala College}\\
  {\small Behala 700060, Kolkata, India}\\
 {\small 3. Department of Mathematics, Raiganj University} \\
{\small  Raiganj 733134, West Bengal, India}}
 \maketitle
\begin{center}
\begin{minipage}{13.cm}

{\bf \begin{center} ABSTRACT \end{center}}
In this paper, we first discuss cohomology and a one-parameter formal deformation theory of Lie-Yamaguti algebras. Next, we study finite group actions on Lie-Yamaguti algebras and introduce equivariant cohomology for Lie-Yamaguti algebras equipped with group actions. Finally, we study an equivariant one-parameter formal deformation theory and show that our equivariant cohomology is the suitable deformation cohomology.

 \smallskip

 \smallskip

 \end{minipage}
 \end{center}
 \normalsize\vskip0.5cm

\section{Introduction}
\def\theequation{\arabic{section}. \arabic{equation}}
\setcounter{equation} {0}

Lie triple systems arose initially in Cartan's study of Riemannian geometry.
Jacobson \cite{Jacobson} first introduced Lie triple systems and Jordan triple systems in connection with problems from Jordan theory
and quantum mechanics, viewing Lie triple systems as subspaces of Lie algebras that are closed relative to the ternary product.
Lie-Yamaguti algebras were introduced by Yamaguti in \cite{Yamaguti1958} to give an algebraic interpretation of the characteristic properties of the torsion and curvature of homogeneous spaces with canonical connection studied in \cite{Nomizu}. He called them generalized Lie triple systems
at first, which were later called ``Lie triple algebras". Recently, they were renamed as ``Lie-Yamaguti algebras". The notion of Lie-Yamaguti algebra is a generalization of both Lie algebra and the Lie triple system. A Lie-Yamaguti algebra is a vector space together with a binary and ternary operation satisfying some compatibility conditions. In \cite{Zhang2015}, the authors studied $(2,3)$-cohomology groups and related deformation theory of Lie-Yamaguti algebras.  

In this paper, we first introduce cohomology groups of all orders for Lie-Yamaguti algebra, and using this cohomology we study a one-parameter formal deformation theory. Our main aim of the present paper is to study Lie-Yamaguti algebras equipped with finite group actions. We define group actions on Lie-Yamaguti algebras following Bredon's \cite{bredon67} idea for the topological $G$-spaces. We provide some examples of equivariant Lie-Yamaguti algebras. Next, we define equivariant cohomology for Lie-Yamaguti algebras. We also develop an equivariant formal deformation theory for Lie-Yamaguti algebras equipped with group actions and show that this equivariant deformation is controlled by our equivariant cohomology. We also obtain some classical deformation-related results following Gerstenhaber \cite{G1, G2, G3, G4, G5} deformation theory for associative algebras.
 \medskip

\section{Preliminaries}
\def\theequation{\arabic{section}.\arabic{equation}}
\setcounter{equation} {0}

Throughout this paper, we work on an algebraically closed field $\mathbb{K}$ of characteristic different from 2 and 3. We  recall some basic definitions regarding  Lie-Yamaguti algebras from \cite{Lin15} and \cite{Yamaguti1958, Yamaguti1969}.

\begin{definition}
A Lie-Yamaguti  algebra  is a triple $(L, [\c, \c], \{\c, \c, \c\})$ in which $L$ is a $\mathbb{K}$-vector space, $[\c, \c]$ a binary operation and $\{\c, \c, \c\}$
a ternary operation on $L$ such that
\begin{eqnarray*}
&&(LY1)~~ [x, y]=-[y,  x],\\
&&(LY2) ~~\{x,y,z\}=-\{y,x,z\},\\
&&(LY3) ~~\circlearrowleft_{(x,y,z)}([[x, y], z]+\{x, y, z\})=0,\\
&&(LY4) ~~\circlearrowleft_{(x,y,z)}(\{[x, y], z, u\})=0,\\
&&(LY5) ~~\{x, y, [u, v]\}=[\{x, y, u\}, v]+[u,\{x, y, v\}],\\
&&(LY6) ~~ \{x, y, \{u, v, w\}\}=\{\{x, y, u\}, v, w\} +\{u,  \{x, y, v\}, w\} +\{u,  v, \{x, y, w\}\},
\end{eqnarray*}
for all $ x, y, z, u, v, w\in L$ and where $\circlearrowleft_{(x,y,z)}$ denotes the sum over cyclic permutation of $x, y, z$.
\end{definition}

A homomorphism between two Lie-Yamaguti  algebras  $(L, [\c, \c], \{\c, \c, \c\})$ and $(L', [\c, \c]', \{\c, \c, \c\}')$ is a linear map $\varphi: L\rightarrow L'$ satisfying
\begin{eqnarray*}
\varphi([x, y])=[\varphi(x), \varphi(y)]', ~~~~~~~~~\varphi(\{x, y, z\})=\{\varphi(x), \varphi(y),\varphi(z)\}'.
\end{eqnarray*}

\begin{definition} Let $(L, [\c, \c], \{\c, \c, \c\})$  be a Lie-Yamaguti  algebra and $V$ be a vector space. A representation of
$L$ on $V$ consists of a linear map $\rho: L\rightarrow$ End($V$) and bilinear maps $D, \theta: L\times L\rightarrow$ End($V$)
such that the following conditions are satisfied:
\begin{eqnarray*}
 &&(R1) ~D(x, y)-\theta(y, x)+\theta(x, y)+\rho([x, y])-\rho(x)\rho(y)+\rho(y)\rho(x)=0,\\
 &&(R2) ~D([x, y], z)+D([y, z],x)+D([z, x], y)=0, \\
 &&(R3)  ~\theta([x, y], z)=\theta(x, z)\rho(y)-\theta(y, z)\rho(x), \\
 &&(R4) ~D(x, y)\rho(z)=\rho(z)D(x, y)+\rho(\{x, y, z\}), \\
 &&(R5)~\theta(x, [y, z])=\rho(y)\theta(x, z)-\rho(z)\theta(x, y),\\
 &&(R6) ~D(x, y)\theta(u, v)=~\theta(u, v)D(x, y)+\theta(\{x, y, u\}, v)+\theta(u, \{x, y, v\}),\\
 &&(R7)  ~\theta(x, \{y, z, u\})=\theta(z, u) \theta(x, y)-\theta(y, u) \theta(x, z)+D(y, z)\theta(x, u),
\end{eqnarray*}
for any $x, y, z, u, v\in L$. In this case, we also call $V$ to be an $L$-module.
\end{definition}

For example, given a Lie-Yamaguti  algebra  $(L, [\c, \c], \{\c, \c, \c\})$, there is a natural adjoint representation on itself. The corresponding
representation maps $\rho, D$ and $\theta$ are given by
\begin{eqnarray*}
\rho(x)(y):= [x, y],~~ D(x, y)z:= \{x, y, z\},~~ \theta(x, y)z:= \{z, x, y\}.
\end{eqnarray*}

Let $V$ be a representation of a Lie-Yamaguti  algebra $(L, [\c, \c], \{\c, \c, \c\})$. Let us define the cohomology groups of $L$ with coefficients
in $V$. Let $f: L\times L \times \c\c\c\times L$ be an $n$-linear map of $L$ into $V$ such that the following conditions are
satisfied:
\begin{eqnarray*}
 f(x_1, \ldots, x_{2i-1}, x_{2i}, \ldots, x_{n})=0, ~~~~\mbox{if}~~~ x_{2i-1}=x_{2i}.
\end{eqnarray*}
The vector space spanned by such linear maps is called an $n$-cochain of $L$, which is denoted by
$C^{n}(L, V)$ for $n\geq 1$.

\begin{definition}
For any $(f, g)\in  C^{2n}(L, V) \times C^{2n+1}(L, V )$ the coboundary operator  $\delta: (f, g)\rightarrow (\delta_I f, \delta_{II} g)$ is a mapping from $C^{2n}(L, V) \times C^{2n+1}(L, V )$ into $C^{2n+2}(L, V) \times C^{2n+3}(L, V )$ defined as follows:
\begin{eqnarray*}
&& (\delta_If)(x_1,x_2,\ldots,x_{2n+2})\\
&=&\rho(x_{2n+1})g(x_1,x_2,\ldots,x_{2n},x_{2n+2})-\rho(x_{2n+2})g(x_1,x_2,\ldots,x_{2n},x_{2n+1})\\
  &&-g(x_1,x_2,\ldots,x_{2n},[x_{2n+1},x_{2n+2}])\\
  &&+\sum_{k=1}^{n}(-1)^{n+k+1}D(x_{2k-1},x_{2k})f(x_{1},\ldots,\widehat{x}_{2k-1},\widehat{x}_{2k},\ldots,x_{2n+2})\\
   &&+\sum_{k=1}^{n}\sum_{j=2k+1}^{2n+2}(-1)^{n+k}f(x_{1},\ldots,\widehat{x}_{2k-1},\widehat{x}_{2k},\ldots,\{x_{2k-1},x_{2k},x_j\},\ldots,x_{2n+2}),\\
&&(\delta_{II}g)(x_1,x_2,\ldots,x_{2n+3})\\
&=&\theta(x_{2n+2},x_{2n+3})g(x_1,\ldots,x_{2n+1})-\theta(x_{2n+1},x_{2n+3})g(x_1,\ldots,x_{2n},x_{2n+1})\\
  &&+\sum_{k=1}^{n+1}(-1)^{n+k+1}D(x_{2k-1},x_{2k})g(x_{1},\ldots,\widehat{x}_{2k-1},\widehat{x}_{2k},\ldots,x_{2n+3})\\
   &&+\sum_{k=1}^{n+1}\sum_{j=2k+1}^{2n+3}(-1)^{n+k}g(x_{1},\ldots,\widehat{x}_{2k-1},\widehat{x}_{2k},\ldots,\{x_{2k-1},x_{2k},x_j\},\ldots,x_{2n+3}).
\end{eqnarray*}
\end{definition}

\begin{proposition}

The coboundary operator defined above satisfies $\delta \circ \delta=0$, that is $\delta_I\circ \delta_I=0$ and $\delta_{II}\circ \delta_{II}=0$.

\end{proposition}

\begin{proof}
The proof is similar to the proof in \cite{MCL}.
\end{proof}

Let $Z^{2n}(L, V) \times Z^{2n+1}(L, V )$ be the subspace of $C^{2n}(L, V) \times C^{2n+1}(L, V )$ spanned by $(f, g)$
such that $\delta(f, g) = 0$ which is called the space of cocycles and $B^{2n}(L, V ) \times B^{2n+1}(L, V ) =
\delta(C^{2n-2}(L, V ) \times C^{2n-1}(L, V ))$ which is called the space of coboundaries.

\begin{definition}
For the case $n\geq 2$, the $(2n, 2n + 1)$-cohomology group of a  Lie-Yamaguti  algebra $(L, [\c, \c], \{\c, \c, \c\})$ with
coefficients in $V$ is defined to be the quotient space:
\begin{eqnarray*}
 H^{2n}(L, V)\times H^{2n+1}(L, V):=(Z^{2n}(L, V) \times Z^{2n+1}(L, V ))/(B^{2n}(L, V) \times B^{2n+1}(L, V )).
\end{eqnarray*}
\end{definition}
In conclusion, we obtain a cochain complex whose cohomology group is called cohomology
group of a  Lie-Yamaguti  algebra $(L, [\c, \c], \{\c, \c, \c\})$ with coefficients in $V$.

 Let $(L, [\c, \c], \{\c, \c, \c\})$  be a Lie-Yamaguti algebra over $\mathbb{K}$ and $\mathbb{K}[[t]]$  the power series ring in one variable $t$ with coefficients in $\mathbb{K}$.
Assume that $L[[t]]$ is the set of formal series whose coefficients are elements of the vector space $L$.

\begin{definition}
Let $(L, [\c, \c], \{\c, \c, \c\})$  be a Lie-Yamaguti algebra.
 A 1-parameter formal deformation of $L$  is a pair of formal power series
 $(f_{t}, g_t)$ of the form
\begin{eqnarray*}
f_{t}=[\c, \c]+\sum_{i\geq 1}f_{i}t^i,~~~~g_{t}=\{\c, \c, \c\}+\sum_{i\geq 1}g_{i}t^i,
\end{eqnarray*}
where each $f_{i}$ is a $\mathbb{K}$-bilinear map $f_{i}: L\times L\rightarrow L$  (extended to be $\mathbb{K}[[t]]$-bilinear) and each $g_{i}$ is a $\mathbb{K}$-trilinear map $g_{i}: L\times L \times L\rightarrow L$  (extended to be $\mathbb{K}[[t]]$-trilinear) such that  $(L[[t]], f_t, g_t)$ is a  Lie-Yamaguti algebra  over $\mathbb{K}[[t]]$.        Set $f_{0}=[\c, \c]$ and $g_0=\{\c, \c, \c\}$, then $f_t$ and $g_t$ can be written as  $f_{t}=\sum_{i\geq 0}f_{i}t^i, g_{t}=\sum_{i\geq 0}g_{i}t^i$, respectively.
\end{definition}
Since $(L[[t]], f_t, g_t)$ is a  Lie-Yamaguti algebra. Then it  satisfies the  following axioms:
\begin{eqnarray}
&&~f_t(x, x)=0\\
&&~g_t(x,x,y)=0\\
&& ~\circlearrowleft_{(x,y,z)}(f_t(f_t(x, y), z)+g_t(x, y, z))=0\\
&& ~\circlearrowleft_{(x,y,z)}g_t(f_t(x, y), z, u)=0\\
&&~g_t(x, y, f_t(u, v))=f_t(g_t(x, y, u), v)+f_t(u,g_t(x, y, v))~~~~~~~\\
&&~ g_t(x, y, g_t(u, v, w))=g_t(g_t(x, y, u), v, w)+g_t(u,  g_t(x, y, v), w) \nonumber\\
&&+g_t(u,  v, g_t(x, y, w)),~~~~~~~~
\end{eqnarray}
for all $ x, y, z, u, v, w\in L$.

\begin{remark}
Equations (2.1)-(2.6) are equivalent to $(n=0,1,2,\cdots)$
\begin{eqnarray}
&&~f_n(x, x)=0\\
&&~g_n(x,x,y)=0\\
&& ~\circlearrowleft_{(x,y,z)}(\sum_{i+j=n}f_i(f_j(x, y), z)+g_n(x, y, z))=0\\
&& ~\circlearrowleft_{(x,y,z)}\sum_{i+j=n}g_i(f_j(x, y), z, u)=0\\
&&~\sum_{i+j=n}g_i(x, y, f_j(u, v))=\sum_{i+j=n}f_i(g_j(x, y, u), v)+f_i(u,g_j(x, y, v))\\
&&~\sum_{i+j=n} g_i(x, y, g_j(u, v, w))=\sum_{i+j=n}g_i(g_j(x, y, u), v, w)\nonumber\\
&&+\sum_{i+j=n} g_i(u,  g_j(x, y, v), w)+ \sum_{i+j=n}g_i(u,  v, g_j(x, y, w))
\end{eqnarray}
for all $ x, y, z, u, v, w\in L$,
respectively. These equations are called the deformation equations of a  Lie-Yamaguti algebra. Eqs.(2.7)-(2.8) imply $(f, g)\in C^2(L, L)\times C^3(L, L)$.
\end{remark}

Let $n=1$ in Eqs. (2.9)-(2.12). Then
\begin{eqnarray*}
&& ~\circlearrowleft_{(x,y,z)}([f_1(x, y), z]+f_1([x, y], z)+g_1(x, y, z))=0;\\
&& ~\circlearrowleft_{(x,y,z)}(\{f_1(x, y), z, u\}+g_1([x, y], z, u))=0;\\
&&~\{x, y, f_1(u, v)\}+g_1(x, y, [u, v])-[g_1(x, y, u), v]\nonumber\\
&&-f_1(\{x,y,u\}, v)-(z, g_1(x, y, z))-f_1(z,\{x, y, u\})=0;\\
&&and\\
&& \{x, y, g_1(u, v, w)\}+g_1(x, y, \{u, v, w\})-\{g_1(x, y, u), v, w\}\nonumber\\
&& -g_1(\{x, y, u\}, v, w)-\{u,  g_1(x, y, v), w\}-g_1(u,  \{x, y, v\}, w)\nonumber\\
&&-\{u,  v, g_1(x, y, w) \}-g_1(u,  v, \{x, y, w\})-\{u,  v, g_1(x, y, w)\}=0.
\end{eqnarray*}

 which imply $(\delta^{2}_{I}, \delta^{2}_{II})(f_1,g_1)=(0, 0)$, i.e.,
 \begin{eqnarray*}
 (f_1,g_1)\in Z^2(L, L)\times Z^3(L, L).
 \end{eqnarray*}
The pair $(f_{1}, g_1)$ is called the infinitesimal  deformation of $(f_{t}, g_t)$.
\medskip

\begin{definition}
Let $(L, [\c, \c], \{\c, \c, \c\})$  be a Lie-Yamaguti algebra.
 Two 1-parameter formal deformations $(f_{t}, g_t)$ and $(f'_{t}, g'_t)$  of $L$ are said to be equivalent,  denoted by $(f_{t}, g_t)\sim (f'_{t}, g'_t)$,
 if there exists a formal isomorphism of  $\mathbb{K}[[t]]$-modules
  \begin{eqnarray*}
\phi_{t}(x)=\sum_{i\geq 0}\phi_{i}(x)t^{i}:(L[[t]],f_{t},g_t)\rightarrow (L[[t]],f'_{t},g'_t),
\end{eqnarray*}
where $\phi_{i}:L \rightarrow L$ is a $\mathbb{K}$-linear map (extended to be $\mathbb{K}[[t]]$-linear) such that
 \begin{eqnarray*}
&& \phi_{0}=id_L, \phi_{t}\circ f_t(x, y)=f'_t(\phi_{t}(x),\phi_{t}(y)),\\
&& \phi_{t}\circ g_t(x, y, z)=g'_t(\phi_{t}(x),\phi_{t}(y), \phi_{t}(z)).
\end{eqnarray*}
\end{definition}
In particular, if $(f_1, g_1)=(f_2, g_2)=\cdots =(0, 0),$ then  $(f_t, g_t)=(f_0, g_0)$ is called the null deformation.
If  $(f_t, g_t)\sim (f_0, g_0)$, then $(f_t, g_t)$ is called the trivial deformation.
If every 1-parameter formal deformation $(f_t, g_t)$ is trivial, then $L$ is called an analytically rigid Lie-Yamaguti algebra.

\begin{theorem}
Let $(f_{t}, g_t)$ and $(f'_{t}, g'_t)$ be
 two equivalent 1-parameter formal deformations  of $L$.
Then the infinitesimal deformations $(f_1, g_1)$ and $(f'_1, g'_1)$
belong to the same cohomology class in  $H^{2}(L,L)\times H^{3}(L,L).$
\end{theorem}

{\bf Proof.}
By the assumption that $(f_1, g_1)$ and $(f'_1, g_1')$ are equivalent, there exists  a formal isomorphism $\phi_{t}(x)=\sum_{i\geq 0}\phi_{i}(x)t^{i}$
 of  $\mathbb{K}[[t]]$-modules satisfying
  \begin{eqnarray*}
\sum_{i\geq 0}\phi_i(\sum_{j\geq 0}f_j(x_1,x_2)t^{j})t^{i}
=\sum_{i\geq 0}f'_i(\sum_{k\geq 0}\phi_{k}(x_1)t^{k},\sum_{l\geq 0}\phi_{l}(x_2)t^{l})t^{i},\\
\sum_{i\geq 0}\phi_i(\sum_{j\geq 0}g_j(x_1,x_2, x_3)t^{j})t^{i}
=\sum_{i\geq 0}g'_i(\sum_{k\geq 0}\phi_{k}(x_1)t^{k},\sum_{l\geq 0}\phi_{l}(x_2)t^{l}, \sum_{m\geq 0}\phi_{m}(x_3)t^{m})t^{i},\\
\end{eqnarray*}
for any $x_1,x_2, x_3\in L.$
Comparing with the coefficients of $t^1$ for two sides of the above equation, we have
    \begin{eqnarray*}
    &&f_1(x_1,x_2)+\phi_{1}([x_1,x_2])=f'_1(x_1,x_2)+[\phi_{1}(x_1),x_2]+[x_1,\phi_{1}(x_2)],\\
&&g_1(x_1,x_2, x_3)+\phi_{1}(\{x_1,x_2, x_3\})\\
&=&g'_1(x_1,x_2, x_3)+\{\phi_{1}(x_1),x_2, x_3\}+\{x_1,\phi_{1}(x_2), x_3\}+\{x_1,x_2, \phi_{1}(x_3)\}.
\end{eqnarray*}
It follows that $(f_1-f'_1, g_1-g'_1)=(\delta^{1}_{I}, \delta^{1}_{II})(\phi_1, \phi_1)\in B^{2}(L,L)\times B^{3}(L,L),$  as desired.
Thus, $(f_1, g_1)$ and $(f'_1, g_1')$ are cohomologous .$\hfill \Box$

\begin{theorem}\label{thm 210}
Let $(L, [\c, \c], \{\c, \c, \c\})$  be a Lie-Yamaguti algebra with $H^{2}(L,L)\times H^{3}(L,L)=0$, then $L$ is analytically rigid.
\end{theorem}

{\bf Proof.}
Let $(f_{t}, g_t)$ be a 1-parameter formal deformation  of $L$. Suppose $f_t=f_0+\sum_{i\geq r}f_it^i$, $g_t=g_0+\sum_{i\geq r}g_it^i$.
Set $n=r$ in Eqs.(2.7)-(2.8), It follows that
\begin{eqnarray*}
(f_r,g_r)\in Z^{2}(L,L)\times Z^{3}(L,L)=B^{2}(L,L)\times B^{3}(L,L).
\end{eqnarray*}
Then there exists $h_r\in C^{1}(L, L)$ such that $(f_r,g_r)=(\delta^1_{I}h_r, \delta^1_{II}h_r)$.

Consider $\phi_t=id_L-h_rt^r,$ then $\phi_t: L\rightarrow L$ is a linear isomorphism. Thus we can define another  1-parameter formal deformation  by $\phi_t^{-1}$ in the form of
\begin{eqnarray*}
f'_t(x, y)=\big( \phi_t^{-1}f_t(\phi_t(x), \phi_t(y)), g'_t(x, y, z)\big)= \phi_t^{-1}g_t(\phi_t(x), \phi_t(y), \phi_t(z)).
\end{eqnarray*}
Set $f'_t=\sum_{i\geq 0}f'_it^i$ and use the fact $\phi_tf'_t(x, y)=f_t(\phi_t(x), \phi_t(y))$, then we have
\begin{eqnarray*}
(id_L-h_rt^r)\sum_{i\geq 0}f'_i(x, y)t^i=(f_0+\sum_{i\geq 0}f_it^i)(x-h_r(x)t^r, y-h_r(y)t^{r}),
\end{eqnarray*}
that is
\begin{eqnarray*}
&&\sum_{i\geq 0}f'_{i}(x, y)t^{i}-\sum_{i\geq 0}h_r\circ f'_{i}(x, y)t^{i+r}\\
&=&f_0(x, y)
   -f_0(h_r(x),y)t^{r}-f_0(x,h_r(y))t^r+f_0(h_r(x),h_r(y))t^{2r}\\
   &&+\sum_{i\geq r}f_{i}(x,y)t^{i}-\sum_{i\geq r}\{f_i(h_r(x),y)-f_i(x, h_r(y))\}t^{i+r}+\sum_{i\geq r}f_i(h_r(x),h_r(y))t^{i+2r}.
\end{eqnarray*}
By the above equation, it follows that
\begin{eqnarray*}
&&f'_0(x, y)=f_0(x, y)=[x, y],\\
&&f'_1(x, y)=f'_1(x, y)=\cdots=f'_{r-1}(x, y)=0,\\
&&f'_r(x, y)-h_r([x, y])=f_r(x, y)-[h_r(x),y]-[x,h_r(y)].
\end{eqnarray*}
Therefore, we deduce
\begin{eqnarray*}
f'_r(x, y)=f_r(x, y)-\delta^1_{I}h_r(x, y)=0.
\end{eqnarray*}
It follows that $f'_t=[\c, \c]+\sum_{i\geq r+1}f'_it^i$. Similarly, we have $g'_t=\{\c, \c, \c\}+\sum_{i\geq r+1}g'_it^i$
 By induction, we have $(f_t, g_t)\sim (f_0, g_0)$, that is, $L$ is   analytically rigid.
$\hfill \Box$

\section{Group action on Lie-Yamaguti algebras}

In this section, we introduce a notion of finite group actions on Lie-Yamaguti algebras and give some examples.

\begin{definition}
Let $(L,[\c,\c],\{\c,\c,\c\})$ be a Lie-Yamaguti algebra and $G$ be a finite group. Then, the group $G$ is said to be acts on $L$ if there exists a map
\[ \psi : G \times L \mapsto L,~~ (g,a) \mapsto \psi (g,a) =ga ,\] satisfying the following conditions :
\begin{enumerate}
\item For each $g \in G$, the map $\psi _g : L \mapsto L , ~~ a \mapsto ga$ is linear.
\item $ea=a$ for all $a\in L $ , where $e$ is the identity element of the group $G$.
\item $g(ha)=(gh)a$ for all $g,h \in G $ and $a\in  L$.
\item $g[a,b] = [ga,gb]$ and $ g \{a,b,c\}=\{ga,gb,gc\}$ for all $g \in G $ and $a,b,c \in L.$
\end{enumerate}
\end{definition}
\begin{example}

J. L. Loday \cite{loday63} defined (left) Leibniz algebra as a $\mathbb{K}$-linear space $L$ with a bilinear map $\cdot : L\times L \to L$ satisfying the following (left) Leibniz identity:
$$a\cdot(b\cdot c) = (a\cdot b)\cdot c + b\cdot  (a\cdot c),~\text{for all}~a,b,c\in L.$$
The notion of Leibniz algebras is a generalization of Lie algebras as in presence of skew-symmetry Leibniz identity is same as Jacobi identity. Let $(L, \cdot)$ be a (left) Leibniz algebra equipped with an action of a finite group $G$ \cite{MS19}. Leibniz algebra $L$ equipped with action of $G$ has an equivariant Lie-Yamaguti algebra structure with respect to the following defined binary and trinary operations:
$$[a, b] := a\cdot b -b \cdot a,~~~\lbrace a,b,c \rbrace := - (a\cdot b)\cdot c,~\text{for all}~a,b,c\in L.$$
\end{example}
\begin{example}
Consider a $2$-dimensional Lie-Yamaguti algebra $\mathfrak{g}$ with a basis $\{e_1,e_2\}$ defined by $[e_1,e_2]=e_1$ , $\{e_1,e_2,e_2\}=e_1$ and for a finite group $G=\{e,g\}$ of order $2$, we define $\psi : G \times \mathfrak{g} \mapsto \mathfrak{g}$ by $\psi(e,a)=a$ for all $a\in \mathfrak{g}$ and $\psi (g,a)=-a$ for all $a\in \mathfrak{g}.$ Here $G$ acts on the Lie-Yamaguti algebra $\mathfrak{g}$.
\end{example}
We use the notation $(G,L)$ to denote a Lie-Yamaguti algebra $L$ equipped with an action of $G$. Let $(G,L)$ and $(G,L^{'})$ be two Lie-Yamaguti algebra equipped with action of $G$. A morphism between $(G,L)$ and $(G,L^{'})$ is a morphism of Lie-Yamaguti algebra $\varphi : L \mapsto L^{'}$ which satisfies $\varphi (ga)=g \varphi (a)$ for all $g \in G$ and $a \in L.$
\begin{proposition}
Let $G$ be a finite group and $L$ be a Lie-Yamaguti algebra. Then $G$ acts on $L$ if and only if there exists a group homomorphism \[\phi : G \rightarrow Iso_{LYA}(L,L), g \mapsto \phi (g)\] from the group $G$ to the group of Lie-Yamaguti algebra isomorphisms from $L$ to $L.$
\end{proposition}
\begin{proof}
For a given action $(L,G)$ , define a map \[\phi : G \rightarrow Iso_{LYA}(L,L), g \mapsto \phi (g)=\psi _g\] Then clearly $\phi$ is a group homomorphism.\
Conversely, for a given group homomorphism \[\phi : G \rightarrow Iso_{LYA}(L,L), g \mapsto \phi (g)\] define a map $G\times L \rightarrow L$ by $(g,a) \mapsto \psi(g,a)=\phi(g)a.$ Clearly, this is an action of $G$ on the Lie-Yamaguti algebra $L.$
\end{proof}

Let $L$ be a Lie-Yamaguti algebra equipped with an action of $G$ and $H$ be a subgroup of $G$. Then $H$-fixed point set is denoted by $L^H$ and defined by \[L^H =\{ a \in L ~ |~ ha = a, ~ \text{for all}~ h \in H\}.\]

\begin{proposition}
Let  $(L,[\c,\c],\{\c,\c,\c\})$ be a Lie-Yamaguti algebra equipped with an action of $G$. For every subgroup $H$ of $G$, $L^H$ is a Lie-Yamaguti sub-algebra of $L$.

\end{proposition}
\begin{proof}
To prove $L^H$ is a Lie-Yamaguti sub-algebra , it is enough to show that $L^H$ is closed under $[\c,\c] ~and ~ \{\c,\c,\c\}$. Let $a,b,c \in L^H$ and $h \in H$. Then  \[h[a,b]=[ha,hb]=[a,b]\] and
\[  h\{a,b,c\}= \{ha,hb,hc\}=\{a,b,c\}.\]
Therefore $[a,b]$ and $\{a,b,c\} \in L^H$.
Thus, $L^H$ is a Lie-Yamaguti sub-algebra.
\end{proof}

\section{Equivariant cohomology of Lie-Yamaguti algebras}

In this section, we introduce an equivariant cohomology of  Lie-Yamaguti algebras equipped with finite group
actions.

\begin{definition}
Let $G$ be a fixed finite group acting on a Lie-Yamaguti algebra  $(L,[\c,\c],\{\c,\c,\c\})$  and $(V, \rho, D, \theta)$ be a representation of $L$. Then the equivariant action of $G$ on $V$ is defined by a map $G \times V \mapsto V $ satisfying the following conditions
\begin{enumerate}
\item $(e,v) \mapsto v$, for all $v \in V$, where $e$ is the identity element of the group $G.$
\item $g_1(g_2v)=(g_1g_2)v$ for all $g_1,g_2 \in G$ and $v \in V.$
\item $g(v_1+cv_2)=gv_1 +c(gv_2)$ for any $g_1,g_2 \in G$ and for any scaler $c.$
\item $\rho(gx)(v)=g( \rho (x) (g^{-1}v))$
\item $D(gx,gy)(v)=g(D(x,y)(g^{-1}v))$
\item $\theta(gx,gy)(v)=g( \theta(x ,y)(g^{-1}v))$
where $x,y,z \in L$ and $v \in V. $
\end{enumerate}

In this case, we say that the Lie-Yamaguti algebra $L$ is equipped with an equivariant action of a finite group $G$ compatible with the  representation of $L$ by a vector space $V$.

\end{definition}
\begin{example}
Let $L$ be a Lie-Yamaguti algebra and $G$ be a group action on $L$. Consider the natural adjoint representation on itself where $\rho , D , \theta $ are given by
\begin{eqnarray*}
\rho(x)(y):= [x, y],~~ D(x, y)z:= \{x, y, z\},~~ \theta(x, y)z:= \{z, x, y\},~ for ~~x,y,z \in L.
\end{eqnarray*}
Now we can see that
\begin{eqnarray*}
\rho(gx)(v)&=&[gx,v] \\
&=&g [x,g^{-1}v] \\
&=&g \rho (x)(g^{-1}(v)).
\end{eqnarray*}
Similarly, $D(gx,gy)(v)=g(D(x,y)(g^{-1}v))$ and  $\theta(gx,gy)(v)=g( \theta(x ,y)(g^{-1}v))$ holds for any $v \in L.$ Thus, $L$ is equipped with the action of $G$ compatible with the adjoint representation.
\end{example}
\begin{definition}
Let  $(L,[\c ,\c],\{\c,\c,\c\})$ be a Lie-Yamaguti algebra and $L$ is equipped with an equivariant action of a finite group $G$ compatible with a representation of $L$ by a vector space $V$. We define
\begin{eqnarray*}
C_{G}^{n}(L, V)= \{f \in C^{n}(L, V) : f(gx_1, \ldots, gx_{i}, \ldots, gx_{n})=gf(x_1, \ldots, x_{i}, \ldots, x_{n})\}
\end{eqnarray*}
Clearly $ C_{G}^{n}(L, V)$ forms a subspace of $C^{n}(L, V)$. The elements of this space is called an equivariant $n$-cochain of $L$, for $n\geq 1$.
\end{definition}
\begin{lemma}
For any $(f,h) \in C_G^{2n}(L,V) \times C_G^{2n+1}(L,V)$ , the coboundary operator $\delta (f,h)=(\delta _If,\delta _{II} h) \in C_G^{2n+2}(L,V) \times C_G^{2n+3}(L,V) $.

\end{lemma}

\begin{proof}
Let $g$ be an element of the group $G.$ Now,
\begin{eqnarray*}
&& (\delta_If)(gx_1,gx_2,\ldots ,gx_{2n+2})\\
&=&\rho(gx_{2n+1})h(gx_1,gx_2,\ldots,gx_{2n},gx_{2n+2})-\rho(gx_{2n+2})h(gx_1,gx_2,\ldots,gx_{2n},gx_{2n+1})\\
  &&-h(gx_1,gx_2,\ldots,gx_{2n},[gx_{2n+1},gx_{2n+2}])\\
  &&+\sum_{k=1}^{n}(-1)^{n+k+1}D(gx_{2k-1},gx_{2k})f(gx_{1},\ldots, \widehat{gx}_{2k-1},\widehat{gx}_{2k},\ldots,g x_{2n+2})\\
   &&+\sum_{k=1}^{n}\sum_{j=2k+1}^{2n+2}(-1)^{n+k}f(g x_{1},\ldots,\widehat{gx}_{2k-1},\widehat{gx}_{2k},\ldots,\{gx_{2k-1},gx_{2k},gx_j\},\ldots,g x_{2n+2}),\\
&=&g[\rho(x_{2n+1})h(x_1,x_2,\ldots,x_{2n},x_{2n+2})]-g[\rho(x_{2n+2})h(x_1,x_2,\ldots,x_{2n},x_{2n+1})]\\
  &&-g[h(x_1,x_2,\ldots,x_{2n},[x_{2n+1},x_{2n+2}])]\\
  &&+g[\sum_{k=1}^{n}(-1)^{n+k+1}D(x_{2k-1},x_{2k})f(x_{1},\ldots,\widehat{x}_{2k-1},\widehat{x}_{2k},\ldots,x_{2n+2})]\\
   &&+g [\sum_{k=1}^{n}\sum_{j=2k+1}^{2n+2}(-1)^{n+k}f(x_{1},\ldots,\widehat{x}_{2k-1},\widehat{x}_{2k},\ldots,\{x_{2k-1},x_{2k},x_j\},\ldots,x_{2n+2})],\\
 &=&g(\delta_If)(x_1,x_2,\ldots,x_{2n+2}).
\end{eqnarray*}
On the other hand,
\begin{eqnarray*}
&&(\delta_{II}h)(gx_1,gx_2,\ldots,g x_{2n+3})\\
&=&\theta(g x_{2n+2},g x_{2n+3})h(g x_1,\ldots,g x_{2n+1})-\theta(g x_{2n+1},g x_{2n+3})h(gx_1,\ldots,gx_{2n},gx_{2n+1})\\
  &&+\sum_{k=1}^{n+1}(-1)^{n+k+1}D(gx_{2k-1},gx_{2k})h(gx_{1},\ldots,\widehat{gx}_{2k-1},\widehat{gx}_{2k},\ldots,gx_{2n+3})\\
   &&+\sum_{k=1}^{n+1}\sum_{j=2k+1}^{2n+3}(-1)^{n+k}h(gx_{1},\ldots,\widehat{gx}_{2k-1},\widehat{gx}_{2k},\ldots,\{gx_{2k-1},gx_{2k},gx_j\},\ldots,gx_{2n+3}).\\
&=&g[\theta(x_{2n+2},x_{2n+3})h(x_1,\ldots,x_{2n+1})]-g[\theta(x_{2n+1},x_{2n+3})h(x_1,\ldots,x_{2n},x_{2n+1})]\\
  &&+g[\sum_{k=1}^{n+1}(-1)^{n+k+1}D(x_{2k-1},x_{2k})h(x_{1},\ldots,\widehat{x}_{2k-1},\widehat{x}_{2k},\ldots,x_{2n+3})]\\
   &&+g[\sum_{k=1}^{n+1}\sum_{j=2k+1}^{2n+3}(-1)^{n+k}h(x_{1},\ldots,\widehat{x}_{2k-1},\widehat{x}_{2k},\ldots,\{x_{2k-1},x_{2k},x_j\},\ldots,x_{2n+3})]. \\
   &=&g(\delta_{II}h)(x_1,x_2,\ldots,x_{2n+3}).
   \end{eqnarray*}
 Therefore, the coboundary operator $\delta (f,h)=(\delta _If,\delta _{II} h) \in C_G^{2n+2}(L,V) \times C_G^{2n+3}(L,V) $.

\end{proof}
Let $Z_G^{2n}(L, V) \times Z_G^{2n+1}(L, V )$ be the subspace of $C_G^{2n}(L, V) \times C_G^{2n+1}(L, V )$ spanned by $(f, h)$
such that $\delta(f, h) = 0$ which is called the space of equivariant  cocycles and $B_G^{2n}(L, V ) \times B_G^{2n+1}(L, V ) =
\delta(C_G^{2n-2}(L, V ) \times C_G^{2n-1}(L, V ))$ which is called the space of equivariant coboundaries.

\begin{definition}
Let $G$ be a fixed finite group, for the case $n\geq 2$, the $(2n, 2n + 1)$- equivariant cohomology group of a  Lie-Yamaguti  algebra $(L, [\c, \c], \{\c, \c, \c\})$ with
coefficients in $V$ is defined to be the quotient space:
\begin{eqnarray*}
 H_G^{2n}(L, V)\times H_G^{2n+1}(L, V):=(Z_G^{2n}(L, V) \times Z^{2n+1}(L, V ))/(B_G^{2n}(L, V) \times B_G^{2n+1}(L, V )).
\end{eqnarray*}
\end{definition}

Thus we have a sub-cochain complex of the cochain complex  defined in Definition 2.5 whose cohomology groups are defined to be  equivariant cohomology group.

\section{Equivariant one-parameter formal deformations of  Lie-Yamaguti algebras}

In this section, we define an equivariant one-parameter formal deformations of Lie-Yamaguti algebras equipped with
an action of a finite group.

\begin{definition}
Let $(L, [\c, \c], \{\c, \c, \c\})$  be a Lie-Yamaguti algebra equipped with an action of a finite group $G$ compatible with a representation of $L$ by a vector space $ V$. Then,
 a equivariant  1-parameter formal deformations of $L$  is a pair of formal power series
 $(f_{t}, g_t)$ of the form
\begin{eqnarray*}
f_{t}=[\c, \c]+\sum_{i\geq 1}f_{i}t^i,~~~~g_{t}=\{\c, \c, \c\}+\sum_{i\geq 1}g_{i}t^i,
\end{eqnarray*}
where each $f_{i}$ is a $\mathbb{K}$-bilinear map $f_{i}: L\times L\rightarrow L$  (extended to be $\mathbb{K}[[t]]$-bilinear) and each $g_{i}$ is a $\mathbb{K}$-trilinear map $g_{i}: L\times L \times L\rightarrow L$  (extended to be $\mathbb{K}[[t]]$-trilinear) such that
\begin{enumerate}
\item  $(L[[t]], f_t, g_t)$ is a  Lie-Yamaguti algebra  over $\mathbb{K}[[t]]$.
\item $f_i(ga,gb)=g(f_i(a,b))$ and $g_i(ga,gb,gc)=g (g_i(a,b,c))$ for all $a,b,c \in L $ , $g \in G$ and  $i\geq 1$.
 \end{enumerate}
\end{definition}
Note that in the above definition if we write  $f_0=[\c,\c] $ and $g_0=\{\c,\c,\c\}$, then $f_t$ and $g_t$ can be written as $f_t=\sum_{i\geq 0}f_it^i$ and $g_t=\sum_{i\geq 0}g_it^i$.\\

\begin{proposition}
Let $(f_t,h_t)$ be an equivariant  1-parameter formal deformation of a Lie-Yamaguti algebra $(L, [\c, \c], \{\c, \c, \c\})$. Then $(f, h)\in C_G^2(L, L)\times C_G^3(L, L)$.
\end{proposition}
\begin{proof}
From the deformation equations of a Lie-Yamaguti algebra we have $(f, g)\in C^2(L, L)\times C^3(L, L)$.
Again, from the definition of equivariant deformation of Lie-Yamaguti algebra we have
\begin{eqnarray*}
&&f_n(ga,gb)=g(f_n(a,b))\\
&&h_n(ga,gb,gc)=g(h_n(a,b,c)),
\end{eqnarray*}
for all $a,b,c \in L$ and for all $g\in G$.
Therefore, $(f, h)\in C_G^2(L, L)\times C_G^3(L, L)$.
\end{proof}

\begin{proposition}
Let $(f_t,h_t)$ be an equivariant  1-parameter formal deformation of a Lie-Yamaguti algebra $(L, [\c, \c], \{\c, \c, \c\})$. Then $(f_1, h_1)\in Z_G^2(L, L)\times Z_G^3(L, L)$. This $(f_1,g_1)$ is called the infinitesimal equivariant deformation of  $(f_t,h_t).$
\end{proposition}
\begin{proof}
 From Remarks 2.7 we have
 \begin{eqnarray*}
 (f_1,g_1)\in Z^2(L, L)\times Z^3(L, L).
 \end{eqnarray*}
Again, from the definition of equivariant deformation of Lie-Yamaguti algebra we have
\begin{eqnarray*}
&&f_1(ga,gb)=g(f_1(a,b))\\
&&h_1(ga,gb,gc)=g(h_1(a,b,c)),
\end{eqnarray*}
for all $a,b,c \in L.$
Therefore, $(f_1, h_1)\in Z_G^2(L, L)\times Z_G^3(L, L)$.

\end{proof}
\begin{definition}
Let $(L, [\c, \c], \{\c, \c, \c\})$  be a Lie-Yamaguti algebra.
 Two equivariant 1-parameter formal deformations $(f_{t}, g_t)$ and $(f'_{t}, g'_t)$  of $L$ are said to be equivalent, denoted by $(f_{t}, g_t)\sim (f'_{t}, g'_t)$,
 if there exists a formal isomorphism of  $\mathbb{K}[[t]]$-modules
  \begin{eqnarray*}
\phi_{t}(x)=\sum_{i\geq 0}\phi_{i}(x)t^{i}:(L[[t]],f_{t},g_t)\rightarrow (L[[t]],f'_{t},g'(t)),
\end{eqnarray*}
where $\phi_{i}:L \rightarrow L$ is a $\mathbb{K}$-linear map (extended to be $\mathbb{K}[[t]]$-linear) such that
 \begin{eqnarray*}
&& \phi_{0}=id_L, \phi_{t}\circ f_t(x, y)=f'_t(\phi_{t}(x),\phi_{t}(y)),\\
&& \phi_{t}\circ g_t(x, y, z)=g'_t(\phi_{t}(x),\phi_{t}(y), \phi_{t}(z)),\\
&& \phi_i (gx)=g\phi_i(x)~~~ \mbox{for all}~~~ i\geq 0~~~\mbox{for all}~~ g \in G.
\end{eqnarray*}
\end{definition}

 In particular, if $(f_1, g_1)=(f_2, g_2)=\cdots =(0, 0),$ then  $(f_t, g_t)=(f_0, g_0)$ is called the null equivariant  deformation.
If  $(f_t, g_t)\sim (f_0, g_0)$, then $(f_t, g_t)$ is called the trivial equivariant  deformation.
If every equivariant 1-parameter formal deformation $(f_t, g_t)$ is trivial, then $L$ is called an analytically rigid equivariant  Lie-Yamaguti algebra.

\begin{theorem}
Let $(f_{t}, g_t)$ and $(f'_{t}, g'_t)$ be
 two equivalent equivariant 1-parameter formal deformations  of $L$.
Then the infinitesimal deformations $(f_1, g_1)$ and $(f'_1, g'_1)$
belong to the same cohomology class in  $H_G^{2}(L,L)\times H_G^{3}(L,L).$
\end{theorem}
\begin{proof}
From Theorem 2.9 we have $(f_1-f'_1, g_1-g'_1)=(\delta^{1}_{I}, \delta^{1}_{II})(\phi_1, \phi_1)\in B^{2}(L,L)\times B^{3}(L,L),$
and for any $x_1,x_2 \in L, g\in G$
\begin{eqnarray*}
&&(f_1-f_1^{'})(gx_1,gx_2)\\
&=&\phi_1([gx_1,gx_2])-[\phi_{1}(gx_1),gx_2]-[gx_1,\phi_{1}(gx_2)],\\
&=&\phi_1(g[x_1,x_2])-[g\phi_{1}(x_1),gx_2]-[gx_1,g\phi_{1}(x_2)],\\
&=&g(\phi_1([x_1,x_2])-[\phi_{1}(x_1),x_2]-[x_1,\phi_{1}(x_2)]),\\
&=&g(f_1-f_1^{'})(x_1,x_2).
\end{eqnarray*}
Again from the proof of Theorem 2.9 for any $x_1,x_2, x_3 \in L, g\in G$
\begin{eqnarray*}
&&(g_1-g_1')(gx_1,gx_2,gx_3)\\
&=&\phi_1(\{gx_1,gx_2,gx_3\})-\{\phi_{1}(gx_1),gx_2, gx_3\}-\{gx_1,\phi_{1}(gx_2), gx_3\}-\{gx_1,gx_2, \phi_{1}(gx_3)\}.\\
&=&\phi_1(g\{x_1,x_2,x_3\})-\{g\phi_{1}(x_1),gx_2, gx_3\}-\{gx_1,g\phi_{1}(x_2), gx_3\}-\{gx_1,gx_2, g\phi_{1}(x_3)\}.\\
&=&g(\phi_1(\{x_1,x_2,x_3\})-\{\phi_{1}(x_1),x_2, x_3\}-\{x_1,\phi_{1}(x_2), x_3\}-\{x_1,x_2, \phi_{1}(x_3)\}).\\
&=&g(g_1-g_1')(x_1,x_2,x_3).
\end{eqnarray*}
Therefore $(f_1-f'_1, g_1-g'_1)\in B_G^{2}(L,L)\times B_G^{3}(L,L)$. This proves the theorem.
\end{proof}
\begin{theorem}
Let $(L, [\c, \c], \{\c, \c, \c\})$  be a Lie-Yamaguti algebra with $H_G^{2}(L,L)\times H_G^{3}(L,L)=0$, then $L$ is equivariantly rigid Lie-Yamaguti algebra.
\end{theorem}
\begin{proof}
Let $(f_{t}, g_t)$ be an equivariant  1-parameter formal deformation  of $L$. Suppose $f_t=f_0+\sum_{i\geq r}f_it^i$, $g_t=g_0+\sum_{i\geq r}g_it^i$.
Then from the proof of Theorem \ref{thm 210}, we have
\begin{eqnarray*}
(f_r,g_r)\in Z^{2}(L,L)\times Z^{3}(L,L)=B^{2}(L,L)\times B^{3}(L,L).
\end{eqnarray*}
Again, since $f_r(gx_1,gx_2)=g(f_r(x_1,x_2))$ and $(g_r(gx_1,gx_2,gx_3))=g(g_r(x_1,x_2,x_3))$,
for any $x_1,x_2,x_3 \in L$ and  $g\in G.$ Hence
\begin{eqnarray*}
(f_r,g_r)\in Z_G^{2}(L,L)\times Z_G^{3}(L,L)=B_G^{2}(L,L)\times B_G^{3}(L,L).
\end{eqnarray*}
Then there exists $h_r\in C_G^{1}(L, L)$ such that $(f_r,g_r)=(\delta^1_{I}h_r, \delta^1_{II}h_r)$.

Consider $\phi_t=id_L-h_rt^r,$ then $\phi_t: L\rightarrow L$ is a linear isomorphism
and for any $g \in G $ and $x \in L$, 
\begin{eqnarray*}
&&\phi_t(gx)\\
&=&(id_L-h_rt^r )(gx)\\
&=&gx-h_r(gx)t^r \\
&=&g(id_L(x))-gh_r(x)t^r\\
&=&g(\phi_t(x)).
\end{eqnarray*}
 Thus, we can define another equivariant 1-parameter formal deformation  by $\phi_t^{-1}$ in the form of
\begin{eqnarray*}
f'_t(x, y)= \phi_t^{-1}f_t(\phi_t(x), \phi_t(y)), g'_t(x, y, z)= \phi_t^{-1}g_t(\phi_t(x), \phi_t(y), \phi_t(z)).
\end{eqnarray*}
Set $f'_t=\sum_{i\geq 0}f'_it^i$ and use the fact $\phi_tf'_t(x, y)=f_t(\phi_t(x), \phi_t(y))$, then we have
\begin{eqnarray*}
(id_L-h_rt^r)\sum_{i\geq 0}f'_i(x, y)t^i=(f_0+\sum_{i\geq 0}f_it^i)(x-h_r(x)t^r, y-h_r(y)t^{r}),
\end{eqnarray*}
that is,
\begin{eqnarray*}
&&\sum_{i\geq 0}f'_{i}(x, y)t^{i}-\sum_{i\geq 0}h_r\circ f'_{i}(x, y)t^{i+r}\\
&=&f_0(x, y)
   -f_0(h_r(x),y)t^{r}-f_0(x,h_r(y))t^r+f_0(h_r(x),h_r(y))t^{2r}\\
   &&+\sum_{i\geq r}f_{i}(x,y)t^{i}-\sum_{i\geq r}\{f_i(h_r(x),y)-f_i(x, h_r(y))\}t^{i+r}+\sum_{i\geq r}f_i(h_r(x),h_r(y))t^{i+2r}.
\end{eqnarray*}
By the above equation, it follows that
\begin{eqnarray*}
&&f'_0(x, y)=f_0(x, y)=[x, y],\\
&&f'_1(x, y)=f'_1(x, y)=\cdots=f'_{r-1}(x, y)=0,\\
&&f'_r(x, y)-h_r([x, y])=f_r(x, y)-[h_r(x),y]-[x,h_r(y)].
\end{eqnarray*}
Therefore, we deduce
\begin{eqnarray*}
f'_r(x, y)=f_r(x, y)-\delta^1_{I}h_r(x, y)=0.
\end{eqnarray*}
It follows that $f'_t=[\c, \c]+\sum_{i\geq r+1}f'_it^i$. Similarly, we have $g'_t=\{\c, \c, \c\}+\sum_{i\geq r+1}g'_it^i$
 By induction, we have $(f_t, g_t)\sim (f_0, g_0)$, that is, $L$ is equivariantly rigid Lie-Yamaguti algebra.
\end{proof}

\begin{center}
 {\bf ACKNOWLEDGEMENT}
 \end{center}

The paper is supported by the NSF of China (No. 12161013) and the general project of Guizhou University of Finance and Economics (No. 2021KYYB16).

\renewcommand{\refname}{REFERENCES}

\end{document}